\documentclass[reqno]{amsart}
\usepackage[T1]{fontenc}
\usepackage{amsmath, amssymb, amsthm, amsfonts}
\usepackage{thmtools}
\usepackage{mathrsfs}

\usepackage{lmodern}
\usepackage{fontawesome5}
\usepackage{booktabs}
\usepackage{caption}
\usepackage{float}

\allowdisplaybreaks

\usepackage{graphicx}
\usepackage[dvipsnames]{xcolor}
\usepackage{mdframed}
\usepackage{tikz-cd}

\usepackage{fancyhdr}
\usepackage[hang,flushmargin]{footmisc}

\usepackage{enumitem}
\usepackage{mathtools}

\usepackage[alphabetic, initials, msc-links]{amsrefs}
\makeatletter
\renewcommand{\BibLabel}{%
    \Hy@raisedlink{\hyper@anchorstart{cite.\CurrentBib}\hyper@anchorend}%
    [\thebib]%
}
\makeatother
\makeatletter
    \renewcommand{\MR}[1]{\ (\href{https://mathscinet.ams.org/mathscinet-getitem?mr=MR#1}{MR#1})}
\makeatother

\usepackage{hyperref}
\hypersetup{
    colorlinks=true,
    linkcolor=-Cerulean,
    filecolor=magenta,
    urlcolor=Cerulean,
    citecolor=Cerulean
}

\graphicspath{ {./Images/} }

\definecolor{darkorange}{RGB}{213,94,0}
\definecolor{darkblue}{RGB}{0,114,178}

\usepackage[capitalize, nameinlink]{cleveref}
\crefname{equation}{Equation}{Equations}
\crefname{prop}{Proposition}{Propositions}
\crefname{section}{Section}{Sections}
\crefname{figure}{Figure}{Figures}

\theoremstyle{definition}
\newtheorem{theorem}{Theorem}[section]
\newtheorem{lemma}[theorem]{Lemma}
\newtheorem{definition}[theorem]{Definition}
\newtheorem{example}[theorem]{Example}
\newtheorem{proposition}[theorem]{Proposition}

\newtheorem*{remark}{Remark}

\newtheorem{conjecture}[theorem]{Conjecture}

\newcommand{\R}{\mathbb{R}}

\DeclareMathOperator{\GL}{\mathrm{GL}}

\DeclareMathOperator{\std}{\mathrm{std}}

\DeclareMathOperator{\Frob}{\mathrm{Frob}}

\DeclareMathOperator{\Stab}{\mathrm{Stab}}
\DeclareMathOperator{\triv}{\mathrm{triv}}

\title[Cohomology of Iwahori Congruence Subgroups]{On the Cuspidal Cohomology of Iwahori Congruence Subgroups of $\mathrm{SL}(3, \mathbb{Z})$}
\author{Zachary Porat}
\address{
Department of Mathematics and Statistics\\
Bucknell University\\
380 Olin Science Building\\
Lewisburg, PA 17837 USA
}
\email{\href{mailto:z.porat@bucknell.edu}{\tt z.porat@bucknell.edu}}
\urladdr{\href{https://zporat.github.io}{\tt https://zporat.github.io}}


\begin{document}
\renewcommand{\thefootnote}{\fnsymbol{footnote}}
\begin{abstract}
We investigate automorphic forms for congruence subgroups of $\mathrm{SL}(3, \mathbb{Z})$ that are Iwahori at $p$.  In particular, we study the cuspidal cohomology of Iwahori congruence subgroups $\mathcal{I}(3, p)$, which are comprised of matrices in $\mathrm{SL}(3, \mathbb{Z})$ that are upper triangular modulo $p$.  In order to work with $\mathcal{I}(3, p)$, we generalize key results from Ash, Grayson, and Green [J.\ Number Theory 19 (1984), pp.\ 412-436]
. For levels $\mathcal{I}(3, p)$ with $p \leq 227$, we found three levels with nonzero cuspidal classes and were able to compute the action of Hecke operators at two of these levels.  These are the first examples of non-essentially-self-dual automorphic representations of trivial cohomological weight that are Steinberg at $p$ appearing at Iwahori level. 
\end{abstract}

\maketitle

\section{Introduction} \label{section:intro}
Let $G$ be a reductive algebraic group, $X$ the associated symmetric space for $G$, and $\Gamma$ an arithmetic subgroup of $G$.  Automorphic forms arise in the cohomology of certain compactifications of $X/\Gamma$.  For example, in the classical setting where $G = \mathrm{GL}(2)$, $X$ is the upper half plane, and $\Gamma$ denotes the congruence subgroup $\Gamma_0(2,N)$, the quotient space $X/\Gamma$ is the modular curve $Y_0(N)$.  Then, one can study the space of weight $2$ cusp forms $S_2(\Gamma_0(2, N))$ by instead investigating specific classes in the cohomology of the Bailey-Borel compactification of $Y_0(N)$, which is the modular curve $X_0(N)$.

In the present article, we let $G = \mathrm{GL}(3)$.  Here, the associated symmetric space is $X = \mathrm{SO}(3) \backslash \mathrm{SL}(3, \mathbb{R})$ with $\mathrm{SL}(3, \mathbb{R})$ acting on the right, a natural generalization of the upper half plane from the $\mathrm{GL}(2)$ setting.  To study \textbf{cuspidal automorphic forms} on $\mathrm{GL}(3)$, which are of primary interest, we examine the cohomology of the Borel-Serre compactification of $X/\Gamma$.  The Borel-Serre compactification of $X/\Gamma$ is a compact manifold ``with corners'' (see \cite{BS73}), which we denote $X_\Gamma^\text{BS}$.  The interior of $X_\Gamma^\text{BS}$ is $X/\Gamma$, but the corners complicate its boundary. For example, the Baily-Borel compactification of $Y_0(N)$ involves adding a finite number of cusp points; by contrast, the cusp points are blown up to lines in the the Borel-Serre compactification of $Y_0(N)$. 

Several articles (e.g.\ \cite{AGG84}, \cite{vGvdKTV97}, \cite{Por25}) study the cohomology of $X_\Gamma^\text{BS}$ in the case of the \textbf{Hecke congruence subgroup}
\[\Gamma = \Gamma_0(3, N) = \{ (a_{ij}) \in \mathrm{SL}(3, \mathbb{Z}) \colon a_{21} \equiv a_{31} \equiv 0 \pmod N\}.\]
Given a cuspidal automorphic form $f$ on $\mathrm{GL}(3)$, there is a corresponding nonzero cohomology class $u_f$ in the \textbf{cuspidal cohomology} $H^3_\text{cusp}(X_{\Gamma}^\text{BS}, \mathbb{C})$, a subspace of the cohomology space 
$H^3(X_{\Gamma}^\text{BS}, \mathbb{C})$.  We call such a $u_f$ a \textbf{cuspidal class}.  If $u_f$ is an eigenclass for all Hecke operators $T_{A(\ell)}$ with complex eigenvalues $a_\ell$ for prime $\ell$, then 
\[
    L_\ell(f, s) = (1 - a_\ell \ell^{-s} + \overline{a_\ell} \ell^{1-2s} - \ell^{3 - 3s})^{-1}
\]
is the local factor of the associated $L$-function for $\ell \nmid N$, when $\Gamma_0(3, N)$ is the level of the form.  For levels $\Gamma_0(3, N)$ with $N < 2000$ for $N$ a rational prime, the aforementioned articles calculate several of these local factors. 

A natural question to ask is for what other congruence subgroups of $\mathrm{SL}(3, \mathbb{Z})$ can one perform analogous computations?  The purpose of this article is to investigate the so-called \textbf{Iwahori congruence subgroup} 
\[\mathcal{I}(n, p) = \{ (a_{ij}) \in \mathrm{SL}(n, \mathbb{Z}) \colon a_{ij} \equiv 0 \pmod p \text{ for all } 1 \leq j < i \leq n\},\]
i.e.\ the matrices in $\mathrm{SL}(n, \mathbb{Z})$ that are upper triangular modulo $p$. Note, we will primarily restrict our attention to $\mathcal{I}(3, p)$ with $p$ a rational prime.  The name of this subgroup arises from the fact that 
\[\mathcal{I}(n,p) = \mathrm{SL}(n, \mathbb{Z}) \cap I_p,\] 
where $I_p$ denotes the Iwahori subgroup of $G(\mathbb{Z}_p)$.  The Iwahori congruence subgroup is natural to consider as its reduction is the standard Borel subgroup of $\mathrm{GL}(3, \mathbb{F}_p)$.

Further, we note that a nonzero eigenclass that arises in the cuspidal cohomology $H^3_\text{cusp}(X_{\mathcal{I}(3,p)}^\text{BS}, \mathbb{C})$ corresponds to a cuspidal automorphic representation of $\mathrm{GL}(3, \mathbb{A})$ whose local component at $p$ is an unramified twist of the Steinberg representation.  Barrera Salazar, Graham, and Williams \cite{BSGW25} recently proved the Greenberg--Benois exceptional zero conjecture for such representations.  However, thus far no explicit examples appear in the literature.  This article serves to provide the first few examples of non-essentially-self-dual automorphic representations of trivial cohomological weight that are Steinberg at $p$ appearing at Iwahori level $\mathcal{I}(3, p)$ for primes $p < 127$ (see \cref{sec:iwahori-comp}).  In particular, for each $p \in \{19, 43\}$, we have found a non-essentially-self-dual automorphic representation of trivial cohomological weight such that $\pi_{p}$ is Steinberg.

The article is structured as follows.  We begin in \cref{sec:prelims} by recalling the necessary background information from \cite{AGG84}.  In order to work with $\mathcal{I}(3, p)$ as opposed to $\Gamma_0(3, p)$, we need to generalize key results from \cite{AGG84}, which is done in \cref{sec:iwahori-results}. We conclude in \cref{sec:iwahori-comp} with our computational results.  For level $\mathcal{I}(3, p)$ with $p \leq 227$, we found previously-unknown nonzero cuspidal classes at levels with $p \in \{19, 43\}$.  These are the first known cuspidal classes appearing at Iwahori level and correspond to the aforementioned non-essentially-self-dual automorphic representations.

We were able to calculate eigenvalues for desired Hecke operators for levels $\mathcal{I}(3, p)$ with $p \in \{19, 43\}$ (see \cref{table:4}), thus confirming that the classes found at these levels are in fact previously-unseen cuspidal classes not coming from the boundary homology or the symmetric square lifting.  However, though the dimension argument provided in \cref{sec:iwahori-results} strongly suggests the existence of previously-unseen cuspidal classes at level $\mathcal{I}(3, 127)$, computational constraints prevented us from finding the eigenvalues at this level. 

\section{Preliminaries} \label{sec:prelims}

Let $G = \mathrm{GL}(3)$ and consider the symmetric space $X = \mathrm{SL}(3)\backslash \mathrm{SL}(3, \mathbb{R})$, where $\mathrm{SL}(3, \mathbb{R})$ acts on the right.  Our primary objective is to study the space of cuspidal automorphic forms for certain congruence subgroups $\Gamma \leq \mathrm{SL}(3, \mathbb{Z})$.  To understand the space of cuspidal automorphic forms, we study a subspace of the cohomology space $H^3(X_\Gamma^\text{BS}, \mathbb{C})$ called the \textbf{cuspidal cohomology}:
\[H^3_\text{cusp}(\Gamma, \mathbb{C}) = \{u_f \in H^3(X_\Gamma^\text{BS}, \mathbb{C}) \colon u_f \text{ with $f$ a cuspidal automorphic form}\}.\]
\noindent We note that the cohomology of $\Gamma$ and $X_\Gamma^{\text{BS}}$ with complex coefficients are canonically isomorphic and can be identified, as done in \cite{AGG84}.  Similarly, their homology spaces can be identified.  Therefore, we will henceforth consider the spaces $H^3(\Gamma, \mathbb{C})$ and $H_3(\Gamma, \mathbb{C})$. 

We have two main objectives.  First, we want to find previously-unseen cuspidal automorphic forms, which are incredibly rare in this setting.  That is, we want to find cuspidal automorphic forms on $\mathrm{GL}(3)$ that are not lifts from the space of weight 2 cusp forms $S_2(\Gamma_0(2, N))$.  We know how to compute the dimension of forms that are lifts (see \cref{sec:iwahori-boundary}), which allows us to compute the dimension of the cuspidal cohomology.  If the dimension of the cuspidal cohomology is different than expected, then we have identified a previously-unseen $\mathrm{GL}(3)$ cuspidal automorphic form.  After finding a previously-unseen form, we next wish to compute the action of \textbf{Hecke operators} on this form, which is done by computing the action on the corresponding cohomology class.  

\subsection{Finding Forms} \label{sec:prelims-finding}
Our process for finding $\mathrm{GL}(3)$ cuspidal automorphic forms requires calculating the dimension of the cuspidal cohomology $H^3_\text{cusp}(\Gamma, \mathbb{C})$.  To compute the dimension of $H^3_\text{cusp}(\Gamma, \mathbb{C})$, we adapt \cite{AGG84}*{Proposition 2.1} as follows:

\begin{proposition}[\cite{AGG84}*{Prop.\ 2.1}] \label{prop:AGG-2.1}
Let $\Gamma$ be a congruence subgroup of $\mathrm{SL}(3, \mathbb{Z})$.  Further, let $X_\Gamma^{\text{BS}}$ be the Borel-Serre compactification of $X/\Gamma$, and let $\partial X_\Gamma^\text{BS}$ denote the boundary of $X_\Gamma^{\text{BS}}$.  Then, 
    \[\dim H_3(\Gamma, \mathbb{C}) = \dim H^3_{\mathrm{cusp}}(\Gamma, \mathbb{C}) + \dim H_3 (\partial X_\Gamma^\text{BS}, \mathbb{C}).\]
\end{proposition}

\noindent Ash, Grayson, and Green \cite{AGG84}*{Section 3} show that homology space $H_3(\Gamma, \mathbb{C})$, which is dual to $H^3(\Gamma, \mathbb{C})$, is isomorphic to a certain vector space of functions that they denote $W(\Gamma)$.  Therefore, we can determine the dimension of $H_3(\Gamma, \mathbb{C})$ by computing the dimension of $W(\Gamma)$.  In turn, we will be able to determine the dimension of $H^3_\text{cusp}(\Gamma, \mathbb{C})$ should we be able to calculate the dimension of $H_3 (\partial X_\Gamma^\text{BS}, \mathbb{C})$.  We start by recalling the necessary details to establish the correspondence between $H_3(\Gamma, \mathbb{C})$ and the vector space $W(\Gamma)$, and then turn our attention to finding the dimension of $H_3 (\partial X_\Gamma^\text{BS}, \mathbb{C})$.

\begin{definition}
    Let $e_1, \ldots, e_n$ be the standard basis of column vectors in $n$-space.  For any ring $R$, we will consider $\sigma_1, \sigma_2, h$ as matrices in $\mathrm{SL}(n, R)$ such that 
    \[\sigma_1 e_1 = e_2, \quad \sigma_1 e_2 = -e_1, \quad \sigma_1 e_i = e_i \quad \text{if } i \geq 3;\]
    \[\sigma_2 e_n = (-1)^{n + 1} e_1, \quad \sigma_2 e_i = e_{i+1} \quad \text{if } 1 \leq i < n;\]
    \[he_1 = e_2, \quad h e_2 = -e_1 - e_2, \quad he_i = e_i \quad \text{if } i \geq 3.\]
    In our setting, where $n = 3$ and $R = \mathbb{Z}$, we have the matrices
    \[\sigma_1 = 
    \begin{pmatrix}
        0 & -1 & 0 \\
        1 & 0 & 0 \\
        0 & 0 & 1
    \end{pmatrix}, 
    \quad \sigma_2 = 
    \begin{pmatrix}
        0 & 0 & 1 \\
        1 & 0 & 0 \\
        0 & 1 & 0
    \end{pmatrix}, 
    \quad h = 
    \begin{pmatrix}
        0 & -1 & 0 \\
        1 & -1 & 0 \\
        0 & 0 & 1
    \end{pmatrix}.
    \]
\end{definition}

\begin{definition}[\cite{AGG84}*{Def.\ 3.1}] \label{def:AGG-3.1}
    For any subgroup $\Gamma$ in $\mathrm{SL}(n, \mathbb{Z})$, we define $W(\Gamma)$ to be the vector space of functions $f \colon \mathrm{SL}(n, \mathbb{Z}) \to \mathbb{C}$ with compact support modulo $\Gamma$ that satisfy:
    \begin{enumerate}
        \item $f(m \gamma) = f(m)$ for all $\gamma \in \Gamma$;
        \item $f(\sigma_1 m) = -f(m)$;
        \item $f(\sigma_2 m) = (-1)^{n + 1} f(m)$;
        \item $f(m) + f(hm) + f(h^2 m) = 0$;
    \end{enumerate}
    for all $m \in \mathrm{SL}(n, \mathbb{Z})$.
\end{definition}

\begin{theorem}[\cite{AGG84}*{Thm.\ 3.2}] \label{thm:AGG-3.2}
Let $n = 2, 3, $ or $4$.  Let $N = \frac{1}{2} n (n-1)$.  Suppose $\Gamma \subseteq \mathrm{SL}(n, \mathbb{Z})$.  Then, there is a natural isomorphism 
\[\Phi \colon W(\Gamma) \to H_N(\Gamma, \mathbb{C}).\]  
If $\Gamma$ is normal in $\mathrm{SL}(n, \mathbb{Z})$, this is an isomorphism of $\mathrm{SL}(n, \mathbb{Z})$-modules, with $W(\Gamma)$ given the structure of an $\mathrm{SL}(n, \mathbb{Z})$-module by the action 
\[(m_2 \cdot f)(m_1) = f(m_1 m_2) \quad \text{ for } m_1, m_2 \in \mathrm{SL}(n, \mathbb{Z}),\]
which factors through $\mathrm{SL}(n, \mathbb{Z}) / \Gamma$.
\end{theorem}

Having established that $W(\Gamma) \simeq H_3(\Gamma, \mathbb{C})$, we now wish to understand how to compute the dimension of $H_3 (\partial X_\Gamma^\text{BS}, \mathbb{C})$, where $\partial X_\Gamma^\text{BS}$ denotes the boundary of $X_\Gamma^{\text{BS}}$. The authors in \cite{AGG84}*{Section 2} give a decomposition for $H_3 (\partial X_\Gamma^\text{BS}, \mathbb{C})$.  Let $T_n$ denote the Tits building for the group $\mathrm{SL}(n, \mathbb{Q})$, which is the simplicial complex with one vertex for every nontrivial subspace of $\mathbb{Q}^n$.  

\begin{proposition}[\cite{AGG84}*{Prop.\ 2.3}] \label{prop:AGG-2.3}
    Suppose $\Gamma \subseteq \mathrm{SL}(3, \mathbb{Z})$ has finite index.  Let $\mathcal{V} = T_3/\Gamma$.  Further, let $\mathcal{P}$ be a set of representatives of orbits of maximal parabolic $\mathbb{Q}$-subgroups of $\mathrm{SL}(3)$ under conjugation by $\Gamma$.  For each $P \in \mathcal{P}$, let $L(P)$ denote the Levi component of $P$.  Let $\Gamma(P)$ denote the projection of $\Gamma$ onto the Levi component.  Then,
    \[H_3(\partial X_\Gamma^\text{BS}, \mathbb{C}) \simeq H^1(\mathcal{V}, \mathbb{C}) \oplus \bigoplus_{P \in \mathcal{P}} H^1_\mathrm{cusp} (\Gamma(P), \mathbb{C}).\]
\end{proposition}

Therefore, computing the dimensions of the components in this direct sum will allow us to compute the dimension of the boundary.  We do this for the Iwahori congruence subgroup $\mathcal{I}(3,p)$ in \cref{sec:iwahori-boundary}.  Once we have found a previously-unseen cuspidal automorphic form at level $\Gamma$, we then wish to compute the action of Hecke operators on the form.  We do so by working on the cohomology side, computing the action on the corresponding cuspidal class. We next review the notion of Hecke operators in the $\mathrm{GL}(3)$ setting, which starts with a discussion about higher rank modular symbols. 

\subsection{Modular Symbols} \label{sec:prelims-mod_symbols}
Let $Q$ be a $3 \times 3$ rational matrix with nonzero columns.  The modular symbol $[Q]$ is an element of $H_1(T_3, \mathbb{Z})$, where $T_3$ is the Tits building for $\mathrm{SL}(3, \mathbb{Q})$.  If $Q \in \mathrm{SL}(3, \mathbb{Z})$, we say the modular symbol $[Q]$ is \textbf{unimodular}.  Modular symbols can be viewed more concretely as a collection of nonzero rational column vectors that enjoy the properties of \cite{AR79}*{Proposition 2.2}.  Note that throughout this article, we will swap column vectors for row vectors and left action for right action. 

These properties allow us to write any modular symbol as the finite sum of unimodular symbols.  Algorithms for the reduction process can be found in \cite{AR79}*{Section 4}, \cite{vGvdKTV97}*{Section 2.10}, and \cite{Gun00}.  Throughout this work, we use an implementation of the algorithm from \cite{vGvdKTV97}*{Section 2.10}.  Moreover, Ash and Rudolph \cite{AR79}*{Definition 3.1} show how we can view $[Q]$ as an element in $H_2(X_\Gamma^{\text{BS}}, \partial X_\Gamma^{\text{BS}} ; \mathbb{C})\simeq H^3(\Gamma, \mathbb{C})$, which leads to the following proposition. 

\begin{proposition}[\cite{AGG84}*{Prop.\ 3.24}] \label{thm:2.7}
    The intersection pairing
    \[\langle -, - \rangle \colon H_2(X_\Gamma^{\text{BS}}, \partial X_\Gamma^{\text{BS}}) \times H_3(X_\Gamma^{\text{BS}}) \to \mathbb{C}\]
    is given by
    \[\langle [Q], \Phi(f) \rangle = f(Q)\]
    for any $Q$ in $\mathrm{SL}(3, \mathbb{Z})$ and $f \in W(\Gamma)$.
\end{proposition}
\noindent \cref{thm:2.7} plays a central role in computing the action of Hecke operators on the cohomology.  We now recall the notion of Hecke operators in the cohomology setting.

\subsection{Hecke Operators} 
\label{sec:prelims-hecke_operators}
For any $A \in \mathrm{GL}(3, \mathbb{Q})$, there is a Hecke operator
\[T_A \colon H^3(\Gamma, \mathbb{C}) \to H^3(\Gamma, \mathbb{C}).\]
Its adjoint operator $T_A^*$ acts on the dual space $H_3(\Gamma, \mathbb{C})$.  Hence, we can evaluate the action of the Hecke operator through the intersection pairing of \cref{thm:2.7}.  More precisely, from \cite{AGG84}*{Section 4, p.\ 426}, we have the equation
\begin{equation} \label{eq:mod-symbols}
    \langle [Q], T_A^* f \rangle = \sum \langle [Q_{i,j}], f \rangle,
\end{equation}
where the $[Q_{i,j}]$ are unimodular symbols such that $[QB_i] = \sum [Q_{i,j}]$ for a fixed $B_i$ in the finite collection of single coset representatives that stem from the decomposition of the double coset $\Gamma A \Gamma$:
\[\Gamma A \Gamma = \coprod_{i=1}^k B_i \Gamma, \quad B_i \in \mathrm{GL}(3, \mathbb{Q}).\]

We will focus our attention on the special Hecke operators $E_\ell = T_{A(\ell)}$ and $F_\ell = T_{B(\ell)}$, with $\ell \neq p$ prime and
\[
A(\ell) =
\begin{pmatrix}
    \ell & 0 & 0 \\
    0 & 1 & 0 \\
    0 & 0 & 1
\end{pmatrix}, \quad
B(\ell) =
\begin{pmatrix}
    \ell & 0 & 0 \\
    0 & \ell & 0 \\
    0 & 0 & 1
\end{pmatrix},
\]
as these specific operators generate the algebra of all Hecke operators $T_A$ acting on $H^3(\Gamma, \mathbb{C})$ by \cite{AGG84}*{Proposition 4.1}.  This algebra is called the \textbf{Hecke algebra} acting on $H^3(\Gamma, \mathbb{C})$.  For a fixed prime $\ell \neq p$, the eigenvalues associated with the special Hecke operators $E_\ell$ and $F_\ell$ are denoted by $e_\ell$ and $f_\ell$ respectively.  

Let $u_f$ be an eigenclass for all Hecke operators $E_\ell$ with complex eigenvalues $e_\ell$ for prime $\ell \neq p$.  In the $\GL(3)$ setting,
\[
    L_\ell(f, s) = (1 - e_\ell \ell^{-s} + \overline{e_\ell} \ell^{1-2s} - \ell^{3 - 3s})^{-1}
\]
is the local factor of the associated $L$-function for $\ell \neq p$, when $\Gamma_0(3, p)$ or $\mathcal{I}(3,p)$ is the level of the form.  We note that $e_\ell$ and $f_\ell$ are conjugate, and so the literature also presents the local factor as 
\[
    L_\ell(f, s) = (1 - e_\ell \ell^{-s} + f_\ell \ell^{1-2s} - \ell^{3 - 3s})^{-1}.
\]
From a computational perspective, it is more efficient to just calculate the action of a single operator $E_\ell$ (resp.\ $F_\ell$) on the space of interest, hence the use of the single eigenvalue $e_\ell$ (resp.\ $f_\ell$).

Alternatively, one can view this data through the lens of Galois representations. Let $G_\mathbb{Q}$ denote the absolute Galois group $\textrm{Gal}(\overline{\mathbb{Q}}/\mathbb{Q})$.  Each Hecke eigenform has a 3-dimensional $G_\mathbb{Q}$ representation attached such that $e_\ell$ is the trace of the Frobenius at $\ell$ and $p$ is the conductor of said representation.  

More generally, let $V$ be a representation of the Hecke algebra.  Suppose $v \in V$ is a simultaneous eigenvector for all Hecke operators $T(\ell, k) = T_{D(\ell, k)}$ with $\ell \neq p$, where $D(\ell, k)$ is the $n \times n$ matrix
\[D(\ell, k) = \begin{pmatrix}
1 & & & & & \\
& \ddots & & & & \\
& & 1 & & & \\
& & & \ell & & \\
& & & & \ddots & \\
& & & & & \ell 
\end{pmatrix},\]
with the first $n-k$ diagonal entries equal to $1$ and the last $k$ diagonal entries equal to $\ell$.  Let $a(\ell, k) \in \mathbb{C}$ denote the eigenvalues such that $T(\ell, k) v = a(\ell, k) v$ for all $\ell \neq p$ prime and $0 \leq k \leq n$.  We say that the continuous Galois representation 
\[\rho \colon G_\mathbb{Q} \to \mathrm{GL}(n, \mathbb{Q}_p)\]
is \textbf{attached} to the eigenform $v$ or that $v$ \textbf{corresponds} to $\rho$ when
\begin{equation} \label{eq:hecke-poly}
\sum_{k=0}^n (-1)^k \ell^{k(k-1)/2} a(\ell, k)X^k = \det (I_{n \times n} - \rho(\Frob_\ell) X),
\end{equation}
for all $\ell \neq p$.  

We will leverage the correspondence to understand data on the left-hand side of the equation, which we call the \textbf{Hecke polynomial}, by considering the representations on the right-hand side of the equation.  The eigenvalues of the image $\rho(\Frob_\ell)$ for a fixed prime $\ell$, the Frobenius element at $\ell$ under the Galois representation $\rho$, are called \textbf{$\ell$-Satake parameters}.  More precisely, an $\ell$-Satake parameter is a semi-simple conjugacy class in $\mathrm{GL}(n, \mathbb{C})$ that corresponds to an unramified irreducible representation $\pi_\ell$, a local component of an automorphic representation $\pi$ (see \cite{Gro96}*{Section 6}).

\section{Main Results} 
\label{sec:iwahori-results}
\subsection{Simplicial Complex}
Recall, our goal is to find cuspidal classes arising in $H^3_\text{cusp}(\mathcal{I}(3,p), \mathbb{C})$.  Previous work has been done for level $\Gamma_0(3, p)$, but in order to move from level $\Gamma_0(3, p)$ to level $\mathcal{I}(3, p)$, we must adapt key results from \cite{AGG84}.  Henceforth, we will use $\mathcal{I}$ to denote $\mathcal{I}(3,p)$.  The primary results to adapt are \cite{AGG84}*{Propositions 3.10, 3.11}, which allow us to compute the dimension of the boundary component of the homology $H_3(\partial X_{\mathcal{I}}^\text{BS}, \mathbb{C})$ using \cref{prop:AGG-2.3}.  

A central component in \cref{prop:AGG-2.3} is the simplicial complex $\mathcal{V} = T_3/\mathcal{I}$, where $T_3$ denotes the Tits building for the group $\mathrm{SL}(3, \mathbb{Q})$.  In order to construct $\mathcal{V} = T_3/\mathcal{I}$, we start by finding a set of representatives of $\mathcal{I}(3, p)$-orbits of maximal and minimal proper parabolic $\mathbb{Q}$-subgroups of $\mathrm{SL}(3)$, which requires the following definition and lemma.

\begin{definition}
    An ordered collection of vectors $\mathcal{B} = \{b_1, \ldots, b_m\}$ is called a \textbf{saturated list} in $\mathbb{Z}^n$ if $\mathcal{B}$ can be extended to a basis of $\mathbb{Z}^n$.
\end{definition}

\begin{lemma} \label{lem:flags}
    Let $\mathbb{F}_p = \mathbb{Z}/p\mathbb{Z}$.  Further, let $\overline{B}$ denote the Borel subgroup of upper triangular matrices in $\overline{G} = \mathrm{SL}(n, \mathbb{F}_p)$.  Let $V = \{v_1, \ldots, v_m\}$ and $W = \{w_1, \ldots, w_m\}$ be two saturated lists in $\mathbb{Z}^n$.  Let $V_i$ and $W_i$ denote the subspaces of $\mathbb{Q}^n$ generated by $\{v_1, \ldots, v_i\}$ and  $\{w_1, \ldots, w_i \}$ respectively. 
    Fix $i_1 < i_2 < \cdots < i_k$ from $\{1, \ldots, m\}$.  Consider the flags $F^V$ and $F^W$, which correspond to a nested sequence of subspaces
    \begin{gather*}
        0 \subsetneq V_{i_1} \subsetneq V_{i_2} \subsetneq \cdots \subsetneq V_{i_k} \subseteq \mathbb{Q}^n, \\
        0 \subsetneq W_{i_1} \subsetneq W_{i_2} \subsetneq \cdots \subsetneq W_{i_k} \subseteq \mathbb{Q}^n.
    \end{gather*}
    Let $\overline{F}^V$ and $\overline{F}^W$ denote the flags 
    \begin{gather*}
        0 \subsetneq \overline{V}_{i_1} \subsetneq \overline{V}_{i_2} \subsetneq \cdots \subsetneq \overline{V}_{i_k} \subseteq \mathbb{F}_p^n, \\
        0 \subsetneq \overline{W}_{i_1} \subsetneq \overline{W}_{i_2} \subsetneq \cdots \subsetneq  \overline{W}_{i_k} \subseteq \mathbb{F}_p^n,
    \end{gather*}
    where $\overline{V}_i = \langle \bar{v}_1, \ldots, \bar{v}_i \rangle$ and $\overline{W}_i = \langle \bar{w}_1, \ldots, \bar{w}_i \rangle$ with 
    $v_j \equiv \bar{v}_j \pmod p$ and $w_j \equiv \bar{w}_j \pmod p$.  Suppose there exists $B_0 \in \overline{B}$ that sends $\overline{F}^V$ to $\overline{F}^W$.  Then, there exists some $\mathcal{I}_0 \in \mathcal{I}(n, p)$ such that $\mathcal{I}_0 \equiv B_0 \pmod p$ and $\mathcal{I}_0$ sends $F^V$ to $F^W$.
\end{lemma}

\begin{proof}
    Since $V$ and $W$ are saturated lists in $\mathbb{Z}^n$, there exists $A \in \mathrm{SL}(n, \mathbb{Z})$ that sends $V_i$ to $W_i$ simultaneously for all $i$.  Therefore, $A$ sends $F^V$ to $F^W$ and $\bar{A} \equiv A \pmod p$ sends $\overline{F}^V$ to $\overline{F}^W$.  Since $B_0$ also sends $\overline{F}^V$ to $\overline{F}^W$, the product $\bar{A}^{-1}B_0$ is in the stabilizer of $\overline{F}^V$, 
    \[\Stab(\overline{F}^V) = \{g \in \overline{G} \colon g \cdot \overline{V}_{i_j} = \overline{V}_{i_j} \text{ for all } j \in \{1, \ldots, k\} \}.\]
    We note that the map 
    \[
    \begin{tikzcd} 
    r \colon \Stab(F^V) \arrow[r, twoheadrightarrow] & \Stab(\overline{F}^V)
    \end{tikzcd}
    \]
    is surjective. Thus, there exists some $C$ in the stabilizer of $F^V$ such that $r(C) = \bar{A}^{-1}B_0$.  Let $\mathcal{I}_0 = AC$.  Then, 
    \[\mathcal{I}_0 = AC \equiv \bar{A} \bar{A}^{-1} B_0 = B_0 \pmod p,\]
    so $\mathcal{I}_0 \in \mathcal{I}(n, p)$ and $\mathcal{I}_0$ sends $F^V$ to $F^W$, as desired.
\end{proof}

\begin{proposition}
\label{lem:A}
    Let $p$ be a rational prime and let $e_1, e_2, e_3$ be the standard basis of column vectors in $\mathbb{Q}^3$.  A set of representatives of $\mathcal{I}(3, p)$-orbits of maximal proper parabolic $\mathbb{Q}$-subgroups of $\mathrm{SL}(3)$ is given by
    \begin{align*}
        P_1 = \mathrm{Stab}\langle e_1 \rangle, && P_1^T  = \mathrm{Stab}\langle e_2, e_3 \rangle, \\
        P_2 = \mathrm{Stab}\langle e_2 \rangle, && P_2^T = \mathrm{Stab}\langle e_1, e_3 \rangle, \\
        P_3 = \mathrm{Stab}\langle e_3 \rangle, && P_3^T = \mathrm{Stab}\langle e_1, e_2 \rangle. 
    \end{align*}
    Moreover, a set of representatives of $\mathcal{I}(3, p)$-orbits of minimal parabolic $\mathbb{Q}$-subgroups of $\mathrm{SL}(3)$ is given by
    \begin{align*}
        P_1 \cap P_2^T, && P_1 \cap P_3^T, \\ 
        P_2 \cap P_3^T, && P_2 \cap P_1^T, \\
        P_3 \cap P_1^T, && P_3 \cap P_2^T.
    \end{align*}
\end{proposition} 

\begin{proof}
We start by noting that for a flag $F$ in $\mathbb{Q}^n$, there exists a saturated list in $\mathbb{Z}^n$ whose associated flag coincides with $F$.  Given that a parabolic subgroup corresponds to the stabilizer of a flag, \cref{lem:flags} implies that to find a set of representatives of $\mathcal{I}(3, p)$-orbits of proper parabolic $\mathbb{Q}$-subgroups of $\mathrm{SL}(3)$, it suffices to find a set of representatives of the Borel orbits of proper parabolic subgroups of $\mathrm{SL}(3, \mathbb{F}_p)$.  

Recall that a maximal proper parabolic subgroup corresponds to the stabilizer of a partial flag.  Thus, a maximal proper parabolic subgroup of $\mathrm{SL}(3, \mathbb{F}_p)$ is exactly the stabilizer of a single, nonzero proper subspace of $\mathbb{F}_p^3$.  There are two types of these subspaces: lines and planes.  In order to understand the orbits of the stabilizers of these subspaces, we invoke the Bruhat decomposition of $\overline{G} = \mathrm{SL}(3, \mathbb{F}_p)$.  

The Bruhat decomposition of $\overline{G}$
is given by 
\begin{equation} \label{eq:bruhat}
    \overline{G} = \overline{B}W\overline{B} = \coprod_{w \in W} \overline{B}w\overline{B},
\end{equation}
where $\overline{B}$ is the Borel subgroup of upper triangular matrices in $\overline{G}$ and $W$ is the Weyl group.  In this case, $W$ is isomorphic to the symmetric group on the set of three elements $S_3$.  

We observe that there is a single $\overline{G}$-orbit of lines, given by $\overline{G} \cdot \langle \bar{e}_1 \rangle$, where $\bar{e}_1 = (1, 0, 0)^T \pmod p$.  \cref{eq:bruhat} allows us to decompose the $\overline{G}$-orbit of lines into the the disjoint union of $\overline{B}$-orbits.  Note that $\langle \bar{e}_1 \rangle$ is fixed by $\overline{B}$.  Therefore, computing $\overline{G} \cdot \langle \bar{e}_1 \rangle$, we find
\begin{align*}
    \overline{G} \cdot \langle \bar{e}_1 \rangle & = \coprod_{w \in W} \overline{B}w\overline{B} \cdot \langle \bar{e}_1 \rangle \\
    & = \coprod_{w \in W} \overline{B}w \cdot \langle \bar{e}_1 \rangle \\
     & =\overline{B} \cdot \langle \bar{e}_1 \rangle \sqcup \overline{B} \cdot \langle \bar{e}_2 \rangle \sqcup \overline{B} \cdot \langle \bar{e}_3 \rangle,
\end{align*}
where
\begin{align*}
\bar{e}_2 & = (0, 1, 0)^T \pmod p, \\
\bar{e}_3 & = (0, 0, 1)^T \pmod p.
\end{align*}
Thus, we have three $\overline{B}$-orbits of lines:
\begin{align*}
    \overline{P}_1 = \Stab\langle \bar{e}_1 \rangle, \\
    \overline{P}_2 = \Stab\langle \bar{e}_2\rangle, \\
    \overline{P}_3 = \Stab\langle \bar{e}_3 \rangle.
\end{align*}

To understand the $\overline{B}$-orbits of planes, we follow a similar strategy. We can again use \cref{eq:bruhat} to decompose the $\overline{G}$-orbit of planes into the the disjoint union of $\overline{B}$-orbits.  There is a single $\overline{G}$-orbit of planes, given by $\overline{G} \cdot \langle \bar{e}_1, \bar{e}_2 \rangle$.  We note that $\langle \bar{e}_1, \bar{e}_2 \rangle$ is fixed by $\overline{B}$.  So, when we compute $\overline{G} \cdot \langle \bar{e}_1, \bar{e}_2 \rangle$, we find 
\begin{align*}
    \overline{G} \cdot \langle \bar{e}_1, \bar{e}_2 \rangle & = \coprod_{w \in W} \overline{B}w\overline{B} \cdot \langle \bar{e}_1, \bar{e}_2 \rangle \\
    & = \coprod_{w \in W} \overline{B}w \cdot \langle \bar{e}_1, \bar{e}_2 \rangle \\
     & =\overline{B} \cdot \langle \bar{e}_1, \bar{e}_2 \rangle \sqcup \overline{B} \cdot \langle \bar{e}_2, \bar{e}_3 \rangle \sqcup \overline{B} \cdot \langle \bar{e}_1, \bar{e}_3 \rangle,
\end{align*}
also giving three $\overline{B}$-orbits:
\begin{align*}
    \overline{P}_1^T = \Stab\langle \bar{e}_2, \bar{e}_3 \rangle, \\
    \overline{P}_2^T = \Stab\langle \bar{e}_1, \bar{e}_3 \rangle, \\
    \overline{P}_3^T = \Stab\langle \bar{e}_1, \bar{e}_2 \rangle.
\end{align*}
Thus, a set of representatives of $\overline{B}$-orbits of maximal proper parabolic subgroups of $\mathrm{SL}(3, \mathbb{F}_p)$ is given by $\{\overline{P}_1, \overline{P}_2, \overline{P}_3, \overline{P}_1^T, \overline{P}_2^T, \overline{P}_3^T\}$.

To determine the minimal parabolic subgroup orbits, we recall that a minimal parabolic subgroup of $\mathrm{SL}(3, \mathbb{F}_p)$ corresponds to the stabilizer of a complete flag.  A complete flag must include a line and a plane such that the line is contained in the plane, e.g. the standard complete flag $F^\text{std}$ is
\[0 \subsetneq \langle \bar{e}_1 \rangle \subsetneq \langle \bar{e}_1, \bar{e}_2 \rangle \subsetneq \mathbb{F}_p^3.\]
Then, the corresponding minimal parabolic subgroup is given by 
\begin{align*}
\Stab(F^\text{std}) & = \{g \in \overline{G} \colon g \cdot \langle \bar{e}_1 \rangle  = \langle \bar{e}_1 \rangle \text{ and } g \cdot \langle \bar{e}_1, \bar{e}_2 \rangle  = \langle \bar{e}_1, \bar{e}_2 \rangle\} \\
& = \Stab \langle \bar{e}_1 \rangle \cap \Stab \langle \bar{e}_1, \bar{e}_2 \rangle \\
& = \overline{P}_1 \cap \overline{P}_3^T.
\end{align*}
We observe that there is a single $\overline{G}$-orbit of complete flags, given by $\overline{G} \cdot F^\text{std}$.  We also see that $\overline{B}$ fixes $F^\text{std}$.  Therefore, 
\begin{align*}
    \overline{G} \cdot F^\text{std} & = \coprod_{w \in W} \overline{B}w\overline{B} \cdot F^\text{std} \\
    & = \coprod_{w \in W} \overline{B}w \cdot F^\text{std}.
\end{align*}
The Weyl elements permute the standard basis vectors, sending $F^\text{std}$ to the complete flags 
\[0 \subsetneq \langle \bar{e}_k \rangle \subsetneq \langle \bar{e}_i, \bar{e}_j \rangle \subsetneq \mathbb{F}_p^3\]
with $k \in \{i, j\}$.  Hence, a set of representatives of $\overline{B}$-orbits of minimal parabolic subgroups is
\begin{align*}
    \overline{P}_1 \cap \overline{P}_2^T, && \overline{P}_1 \cap \overline{P}_3^T, \\ 
    \overline{P}_2 \cap \overline{P}_3^T, && \overline{P}_2 \cap \overline{P}_1^T, \\
    \overline{P}_3 \cap \overline{P}_1^T, && \overline{P}_3 \cap \overline{P}_2^T.
\end{align*}

The final step is to invoke \cref{lem:flags} to lift the representatives of $\overline{B}$-orbits of proper parabolic subgroups of $\mathrm{SL}(3, \mathbb{F}_p)$ to representatives of $\mathcal{I}(3,p)$-orbits of proper parabolic $\mathbb{Q}$-subgroups, giving the desired result. \qedhere


\end{proof}


From \cref{lem:A}, we can now construct $\mathcal{V} = T_3/\mathcal{I}$.  We note that the vertices of $\mathcal{V}$ are given by the $\mathcal{I}(3, p)$-orbits of maximal proper parabolic subgroups and the edges are drawn based on the $\mathcal{I}(3, p)$-orbits of minimal parabolic subgroups.  This construction results in the simplicial complex shown in \cref{fig:Tits}.  Before stating the main theorem, we recall the following definition from \cite{AGG84}. 

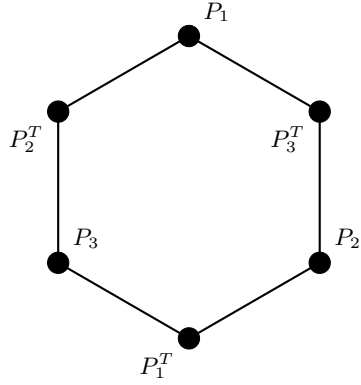
\begin{figure}[h] 
\centering
\begin{tikzpicture}[scale=1, every node/.style={font=\small}]
  \def\R{2cm}        
  \def\colorA{black}   
  \def\colorB{black}  
  \def\vertsize{4pt} 

  \foreach \i in {1,...,6}{
    \coordinate (P\i) at ({90 - (\i-1)*60}:\R);
  }

  \draw[thick] (P1) \foreach \i in {2,...,6} { -- (P\i) } -- cycle;

  \foreach \i in {1,...,6}{
    \ifodd\i
      \def\fillcolor{\colorA}
      \pgfmathtruncatemacro{\n}{(\i+1)/2} 
      \filldraw[fill=\fillcolor, draw=black] (P\i) circle (\vertsize)
        node[above right=2pt] {$P_{\n}$};
    \else
      \def\fillcolor{\colorB}

      \ifnum\i=2
        \def\n{3}
      \fi
      \ifnum\i=4
        \def\n{1}
      \fi
      \ifnum\i=6
        \def\n{2}
      \fi
      \filldraw[fill=\fillcolor, draw=black] (P\i) circle (\vertsize)
        node[below left=2pt] {$P_{\n}^{T}$};
    \fi
  }
\end{tikzpicture}
\caption{The simplicial complex $\mathcal{V} = T_3/\mathcal{I}$.}
\label{fig:Tits}
\end{figure}

\begin{definition}[\cite{AGG84}*{Def.\ 3.10}] \label{def:1}
    Let $\Delta(p)$ be the subgroup of $\GL(2, \mathbb{Z})$ generated by $\Gamma_0(2, p)$ and the matrix
    \[s = \begin{pmatrix}
        -1 & 0 \\
        0 & 1
    \end{pmatrix}.\]
\end{definition}

\begin{theorem} \label{prop:B}
    $H_3(\partial X_{\mathcal{I}}^\text{BS}, \mathbb{C}) \simeq H^1(\mathcal{V}, \mathbb{C}) \oplus [H^1_\mathrm{cusp} \big(\Delta(p), \mathbb{C}\big)]^6$.
\end{theorem}

\begin{proof}
We apply \cref{prop:AGG-2.3}.  In contrast to the $\Gamma_0(3, p)$ setting (see \cite{AGG84}*{Proposition 3.10}), we observe that $H^1(\mathcal{V}, \mathbb{C})$ is nontrivial from \cref{fig:Tits}.  The next elements that we need to understand are $H^1_\mathrm{cusp} (\mathcal{I}(P), \mathbb{C})$ for each $P$ given in \cref{lem:A}, the representatives of $\mathcal{I}$-orbits of maximal parabolic subgroups.  Let $\mathcal{I}(P)$ be the projection of $\mathcal{I}$ onto the Levi component of $P$, denoted $L(P)$.  We start by noting that $\mathcal{I}(P) = \mathcal{I} \cap L(P)$.  Moreover, we see that $L(P_i) = L(P_i^T)$ for $i \in \{1, 2, 3\}$.  In particular, we have
\[L(P_1) = L(P_1^T) = 
\left\{\begin{pmatrix}
    * & 0 & 0 \\
    0 & * & * \\
    0 & * & * 
\end{pmatrix} \right\}, \]

\[L(P_2) = L(P_2^T) = 
\left\{\begin{pmatrix}
    * & 0 & * \\
    0 & * & 0 \\
    * & 0 & * 
\end{pmatrix} \right\}, \]

\[L(P_3) = L(P_3^T) = 
\left\{\begin{pmatrix}
    * & * & 0 \\
    * & * & 0 \\
    0 & 0 & * 
\end{pmatrix} \right\}.\]
Therefore, 
\[\mathcal{I}(P_1) = \mathcal{I} \cap L(P_1) = \mathcal{I} \cap L(P_1^T) = \mathcal{I}(P_1^T) = \left\{ \begin{pmatrix}
    * & 0 & 0 \\
    0 & \textcolor{-Cerulean}{*} & \textcolor{-Cerulean}{*} \\
    0 & \textcolor{-Cerulean}{0} & \textcolor{-Cerulean}{*} 
\end{pmatrix} \pmod p \right\},\]
\[\mathcal{I}(P_2) = \mathcal{I}\cap L(P_2) = \mathcal{I} \cap L(P_2^T) = \mathcal{I}(P_2^T) = \left\{ \begin{pmatrix}
    \textcolor{-Cerulean}{*} & 0 & \textcolor{-Cerulean}{*} \\
    0 & * & 0 \\
    \textcolor{-Cerulean}{0} & 0 & \textcolor{-Cerulean}{*}
\end{pmatrix} \pmod p \right\},\]
\[\mathcal{I}(P_3) = \mathcal{I} \cap L(P_3) = \mathcal{I} \cap L(P_3^T) = \mathcal{I}(P_3^T) = \left\{ \begin{pmatrix}
    \textcolor{-Cerulean}{*} & \textcolor{-Cerulean}{*} & 0 \\
    \textcolor{-Cerulean}{0} & \textcolor{-Cerulean}{*} & 0 \\
    0 & 0 & *
\end{pmatrix} \pmod p \right\}.\]
For each $\mathcal{I}(P)$, the orange entries highlight the copy of $\Delta(p)$ (see \cite{AGG84}*{Def.\ 3.10}) embedded in the corresponding Levi component.  Hence, $\mathcal{I}(P) \simeq \Delta(p)$ for all six $\mathcal{I}(3,p)$-orbits of maximal proper parabolic subgroups given in \cref{lem:A}.  This completes the proof. 
\end{proof}

\begin{lemma} 
    The dimension of the homology of the boundary of $X^\text{BS}_\mathcal{I}$ is given by the following formula:
    \begin{equation}\label{eq:boundary}
    \dim H_3(\partial X_\mathcal{I}^\text{BS}, \mathbb{C}) = 1 + 6\,\dim S_2(\Gamma_0(2, p)).
    \end{equation}
\end{lemma}

\begin{proof}
By \cref{fig:Tits}, we see that $\mathcal{V}$ is a connected graph with six edges and six vertices.  Thus, the dimension of $H^1 (\mathcal{V}, \mathbb{C})$, given by the first Betti number of $\mathcal{V}$, is 
    \[\dim H^1 (\mathcal{V}, \mathbb{C}) = b_1 = \# \text{Edges} - \#\text{Vertices} + 1 = 1.\]
Furthermore, by \cite{AGG84}*{Proposition 3.11}, we have
    \[\dim H^1_\mathrm{cusp} (\Delta(p), \mathbb{C}) = \dim S_2(\Gamma_0(2, p)).\]
Combining these results with \cref{prop:B} gives the following dimension formula for the homology of the boundary of $X_\mathcal{I}^\text{BS}$:
\begin{equation*} 
    \dim H_3(\partial X_\mathcal{I}^\text{BS}, \mathbb{C}) = 1 + 6\,\dim S_2(\Gamma_0(2, p)). \qedhere
\end{equation*}
\end{proof}
\noindent With \cref{eq:boundary} in hand, we can now invoke \cref{prop:AGG-2.1} to find the dimension of the cuspidal cohomology in the $\mathcal{I}(3, p)$ setting: 
\[\dim H^3_\text{cusp}(\mathcal{I}(3,p), \mathbb{C}) = \dim H_3(\mathcal{I}(3,p), \mathbb{C}) - (1 + 6\,\dim S_2(\Gamma_0(2, p))).\]

Being able to compute the dimension of the cuspidal cohomology is only the first step.  We then wish to compute the action of Hecke operators on any previously-unknown classes appearing in the cuspidal cohomology.  However, unlike the $\Gamma_0(3, p)$ setting where we can compute characteristic polynomials of the Hecke operators $E_\ell$ and $F_\ell$ (see \cref{sec:prelims-hecke_operators}) on the cuspidal cohomology directly, in the $\mathcal{I}(3, p)$ setting, we have thus far been unable to generalize the necessary results to make the same process possible (see the author's thesis for partial progress towards this goal).  

As a result, we have to compute the action of $E_\ell$ (resp.\ $F_\ell$) on the full homology space $H_3(\mathcal{I}(3, p), \mathbb{C})$.  Therefore, the resulting characteristic polynomial, which we will denote $\varphi_\ell^\text{full}$, is comprised of factors that come from the boundary cohomology, which we will denote $\varphi_\ell^\text{bound}$, as well as factors from the cuspidal cohomology, which we will denote $\varphi_\ell^\text{cusp}$, i.e.\ $\varphi_\ell^\text{full} = \varphi_\ell^\text{bound} \varphi_\ell^\text{cusp}$. In \cref{sec:iwahori-boundary}, we explain how to find the factors in $\varphi_\ell^\text{bound}$ and in \cref{sec:iwahori-cusp}, we discuss how to isolate any previously-unseen classes from $\varphi_\ell^\text{cusp}$.

\subsection{Understanding the Boundary}
\label{sec:iwahori-boundary}
To isolate the factors in $\varphi_\ell^\text{bound}$, we examine the polynomial $\varphi_\ell^\text{full}$ for factors that correspond to self-dual cusp forms in $S_2(\Gamma_0(2, p))$.  The roots of these factors can be written in terms of the Hecke eigenvalues associated with a self-dual cusp form, which follows from work of Harder \cite{Har91}.  One can explicitly determine these factors by examining how the Satake parameters from the $\mathrm{GL}(2)$ cusp forms embed into $\mathrm{GL}(3)$. 

Consider the embeddings $\psi_1$ and $\psi_2$ from $\mathrm{GL}(2)$ into $L(P)$, the Levi component of the maximal parabolic subgroup $P = \Stab\langle e_1, e_2 \rangle$, corresponding to the $3$-dimensional representations of $\mathrm{GL}(2)$
\begin{align*}
\hat{\rho}_1 & = (\det \otimes \std) \oplus \triv, \\
\hat{\rho}_2 & = \std \oplus \det{^2}.
\end{align*}
By the theory of parabolic induction, an automorphic representation of $\mathrm{GL}(2)$ and a character of $\mathrm{GL}(1)$, viewed as a representation of $L(P)$, give rise to an automorphic representation of $\mathrm{GL}(3)$.  The parabolic induction associated with $\psi_1$ and $\psi_2$ take a weight $2$ cusp form to an automorphic representation of $\mathrm{GL}(3)$ with trivial weight (see \cite{Clo90}).  

Let $\rho_0$ denote the $2$-dimensional Galois representation attached to a self-dual cusp form $\hat{f} \in S_2(\Gamma_0(2, p))$.  Let $s_\ell, t_\ell$ be $\ell$-Satake parameters for $\hat{f}$, i.e.\ $s_\ell, t_\ell \in \mathbb{R}$ are the eigenvalues of $\rho_0(\Frob_\ell)$, the image of the Frobenius element at $\ell$ under the Galois representation $\rho_0$ associated to $\hat{f}$.  Applying the embeddings $\psi_1$ and $\psi_2$ to the diagonal matrix 
\[\begin{pmatrix}
    s & 0 \\
    0 & t
\end{pmatrix},\]
we find
\[\begin{pmatrix}
    s &  \\
     & t
    \end{pmatrix}
    \xmapsto{\quad \psi_1 \quad}
    \begin{pmatrix}
    s^2 t & &  \\
    & s t^2 & \\
    & & 1
\end{pmatrix},\]
\[\begin{pmatrix}
    s &  \\
     & t
    \end{pmatrix}
    \xmapsto{\quad \psi_2 \quad}
    \begin{pmatrix}
    s & &  \\
    & t & \\
    & & (st)^2
\end{pmatrix}.\]
Knowing the embeddings explicitly, we can calculate the right-hand side of \cref{eq:hecke-poly}, thus allowing us to compute the eigenvalues appearing on the left-hand side, which are roots of the factors appearing in $\varphi_\ell^\text{bound}$.

For the sake of simplicity, since we only consider $\mathrm{GL}(2)$ and $\mathrm{GL}(3)$, we will label the eigenvalues once and for all.  Recall that $a(\ell, k)$ is the eigenvalue of the Hecke operator $T(\ell, k)$, which corresponds to the $n \times n$-matrix with the first $n-k$ diagonal entries equal to $1$ and the last $k$ diagonal entries equal to $\ell$ (see \cref{sec:prelims-hecke_operators}).  In the $\mathrm{GL}(2)$ setting, we let $a_\ell$ denote $a(\ell, 1)$.   In the $\mathrm{GL}(3)$ setting, we let $e_\ell = a(\ell, 1)$ and $f_\ell = a(\ell, 2)$.  Note $e_\ell$ and $f_\ell$ are so named to correspond to the special Hecke operators $E_\ell$ and $F_\ell$ described in \cref{sec:prelims-hecke_operators}. 

Let us start by using \cref{eq:hecke-poly} for $n = 2$ to write $a_\ell$ and $\ell$ in terms of the Satake parameters $s, t$:
\begin{align*}
1 - a_\ell X + \ell X^2 & = \det(I_{2 \times 2} - \rho_0(\Frob_\ell)X)\\
& = (1 - sX)(1 - tX) \\
& = 1 - (s + t)X + st X^2.
\end{align*}
Therefore, $st = \ell$ and $s + t = a_\ell$. 

We can now use \cref{eq:hecke-poly} for $n = 3$ to calculate $e_\ell$ and $f_\ell$ in terms of $a_\ell$ and $\ell$ for the representations 
\begin{align*}
\rho_1 & = \psi_1 \circ \rho_0 \colon G_\mathbb{Q} \to \mathrm{GL}(3, \mathbb{Q}_p),\\
\rho_2 & = \psi_2 \circ \rho_0 \colon G_\mathbb{Q} \to \mathrm{GL}(3, \mathbb{Q}_p).
\end{align*}
Starting with $\rho_1$, we compute:
\begin{align*}
    1 - e_\ell X + \ell f_\ell X^2 - \ell^3 X^3 & = \det(I_{3 \times 3} - \rho_1(\Frob_\ell)X) \\
    & = (1 - s^2 t X) (1- s t^2 X) (1 - X) \\
    & = 1 - (s^2 t+ s t^2 + 1)X + (s^2 t + s t^2 + s^3 t^3)X^2 - 
    s^3 t^3 X^3.
\end{align*}
First, we solve for $e_\ell$, leveraging the fact that $a_\ell = s + t$ and $\ell = st$: 
\begin{align*}
e_\ell & = s^2 t + st^2 + 1 \\
& = (st)(s + t) + 1 \\
& = \ell a_\ell + 1.
\end{align*}
Similarly, we can solve for $f_\ell$:
\begin{align*}
f_\ell & = \frac{s^2 t + s t^2 + s^3 t^3}{\ell} \\
& = \frac{(s t)(s + t + s^2 t^2)}{\ell} \\
& = a_\ell + \ell^2.
\end{align*}
We can then perform the same process for $\rho_2$:
\[
    1 - e_\ell X + \ell f_\ell X^2 - \ell^3 X^3  = \det(I_{3 \times 3} - \rho_2(\Frob_\ell)X),
\]
which yields 
\[
e_\ell = a_\ell + \ell^2, \qquad f_\ell = 1 + \ell a_\ell.
\]
We note that the value for $e_\ell$ attached to $\rho_1$ is the same as the value for $f_\ell$ attached to $\rho_2$.  Similarly, we see that the value for $f_\ell$ attached to $\rho_1$ is the same as the value for $e_\ell$ attached to $\rho_2$.  Further, recall that the embeddings $\psi_1$ and $\psi_2$ are defined with respect to the parabolic subgroup $P = \Stab\langle e_1, e_2 \rangle$.  By \cref{lem:A}, we know that there are six $\mathcal{I}(3, p)$-orbits of maximal parabolic subgroups, occurring in three pairs $P, P^T$.  As seen in the proof of \cref{prop:B}, each pair shares a Levi component, i.e.\ $L(P) = L(P^T)$.  Computational evidence in \cref{sec:iwahori-hecke_comp} shows that we have three copies of the factors arising from $\psi_1$ and $\psi_2$.  We attribute these copies to the three orbits of Levi components for the maximal parabolic subgroups.

In \cref{sec:iwahori-hecke_comp}, we observe from computational data that the polynomial $\varphi_\ell^\text{bound}$ also contains the linear factor $(T - (\ell^2 + \ell + 1))$.  We note that there are exactly $\ell^2 + \ell + 1$ single coset representatives $B_i$ for $E_\ell$ (resp.\ $F_\ell$).  From \cref{prop:B}, we know that this linear factor must be associated to $H^1(\mathcal{V}, \mathbb{C})$, but do not yet have an explanation as to why the root is precisely the number of coset representatives for the Hecke operators.  In fact, the value $\ell^2 + \ell + 1$ being preserved is somewhat surprising given that during the computation, the reduction algorithm is used to turn the $\ell^2 + \ell + 1$ modular symbols $[B_i]$ into sums of larger collections of unimodular symbols. 

\subsection{Additional Cuspidal Factors}
\label{sec:iwahori-cusp}
We now shift our focus to understanding the factors of $\varphi_\ell^\text{cusp}$ coming from forms of lower level and forms from smaller-rank groups.  The first such factor comes from the symmetric square lifting, originally studied in work of Gelbart and Jacquet \cite{GJ76}.  In particular, we are interested in the $3$-dimensional Galois representation
\[\rho_3 = \psi_3 \circ \rho_0 \colon G_\mathbb{Q} \to \mathrm{GL}(3, \mathbb{Q}_p),\]
where $\psi_3$ is the symmetric square homomorphism, which is the adjoint representation of $\mathrm{GL}(2)$ on the $2 \times 2$ matrices of trace $0$ twisted by the determinant (see \cite{AS86}*{Section 3.4}):
\[
\begin{pmatrix}
    s &  \\
     & t
    \end{pmatrix}
    \xmapsto{\quad \psi_3 \quad}
    \begin{pmatrix}
    s^2 & &  \\
    & st & \\
    & & t^2
\end{pmatrix}.
\]
Once again, we are letting $s, t \in \mathbb{R}$ denote Satake parameters for a self-dual cusp from in $S_2(\Gamma_0(2, p))$. 

In \cite{AT99}, the authors work out an expression for $e_\ell$ and $f_\ell$ in terms of $a_\ell$ and $\ell$.  We give additional details below, specifying to trivial weight, using the same strategy as we did for the boundary factors. Inputting $\rho_3$ into \cref{eq:hecke-poly}, we have
\[
1 - e_\ell X + \ell f_\ell X^2 - \ell^3 X^3  = \det(I_{3 \times 3} - \rho_3(\Frob_\ell)X).
\]
Invoking the fact that $a_\ell = s + t$ and $\ell = st$, we solve for $e_\ell$ and $f_\ell$ to find:
\[e_\ell = f_\ell = a_\ell^{\ 2} - \ell.\]

In addition to the factor coming from the symmetric square lifting, $\varphi_\ell^\text{cusp}$ contains a factor for each cusp form of level $\Gamma_0(3, p)$, but only for $p$ at which such a cusp form exists.  There is a natural embedding from $\mathcal{I}(3, p)$ into $\Gamma_0(3, p)$.  This induces a map on the cohomology, and so we expect to see the characteristic polynomial of $E_\ell$ (resp.\ $F_\ell$) on $H^3_\text{cusp}(\Gamma_0(3, p), \mathbb{C})$ appear as a factor in $\varphi_\ell^\text{cusp}$.  Computational evidence shows that three copies of this characteristic polynomial appear (see \cref{sec:iwahori-hecke_comp}).  This leads to the following conjecture. 

\begin{conjecture}
    Let $\pi$ be a unitary, admissible, $\mathbb{C}$-representation of $\mathrm{GL}(3, \mathbb{Q}_p)$ with the space of vectors fixed by $\Gamma_0(3,p)$, denoted $\pi^{\Gamma_0(3, p)}$, $1$-dimensional.  Then, $\pi^{\mathcal{I}(3, p)}$, the space of vectors fixed by $\mathcal{I}(3,p)$, is $3$-dimensional. 
\end{conjecture}

\section{Computational Methods} \label{sec:iwahori-comp}
The code for the computations described in the following section is implemented in the computer algebra system \textsc{Magma} \cite{Magma}.  The scripts and related data can be found in the GitHub repository \cite{Iwahori-Code}.  The computations were performed on Wesleyan University's High Performance Compute Cluster. 

\subsection{Cuspidal Cohomology Dimension Computation}
\label{sec:iwahori-dim_comp}
Recall that in order to compute the dimension of the cuspidal cohomology $H_\text{cusp}^3 (\mathcal{I}(3,p), \mathbb{C})$, we can leverage \cref{prop:AGG-2.1}.  Restating the proposition to include the formula for the dimension of the boundary component of $X_{\mathcal{I}(3, p) }^\text{BS}$ given in \cref{eq:boundary} yields: 
\[\dim H_\text{cusp}^3 (\mathcal{I}(3,p), \mathbb{C}) = \dim H_3(\mathcal{I}(3,p), \mathbb{C}) - \big(1 + 6\,\dim S_2(\Gamma_0(2, p))\big).\]
The dimension of the space of weight 2 cusp forms $S_2(\Gamma_0(2, p))$ is known.  Therefore, we only need to compute the dimension of the homology $H_3(\mathcal{I}(3,p), \mathbb{C})$.  To compute the dimension of $H_3(\mathcal{I}(3,p), \mathbb{C})$, we first construct a matrix $M$ whose kernel is the vector space $W(\mathcal{I}(3,p))$, which is isomorphic to $H_3(\mathcal{I}(3,p), \mathbb{C})$ by \cref{thm:AGG-3.2}.  We then compute the corank of $M$.

Unlike the $\Gamma_0(3, p)$ setting, however, there are lifts from $\mathrm{GL}(2)$ cusp forms and lower-level $\mathrm{GL}(3)$ cusp forms that contribute to the dimension of $H_\text{cusp}^3 (\mathcal{I}(3,p), \mathbb{C})$, as discussed in \cref{sec:iwahori-cusp}.  In \cref{sec:iwahori-hecke_comp}, we explain how we identified these additional contributions in practice and how we isolated any new classes arising in $H_\text{cusp}^3 (\mathcal{I}(3,p), \mathbb{C})$. As a preview, to find the dimension of the new classes, we computed 
\begin{equation} \label{eq:iwahori-new}
\begin{aligned}
\dim H^3_{\text{new}}(\mathcal{I}(3,p), \mathbb{C}) & = \dim H_3(\mathcal{I}(3,p), \mathbb{C})  - 1 - 7 \dim S_2(\Gamma_0(2, p))\\
& \qquad - 3 \dim H^3_{\text{cusp}}(\Gamma_0(3,p), \mathbb{C}).
\end{aligned}
\end{equation}

To build the matrix $M$, we start by fixing a set of coset representatives for $\mathrm{SL}(3, \mathbb{Z}) / \mathcal{I}(3, p)$.  Let $\mathcal{A}$ denote this set and let $N = (p^2 + p +1)(p+1)$ denote the size of $\mathcal{A}$.  In our construction, we let
\[
\mathcal{A} = \left\{
\begin{pmatrix}
    0 & 0 & 1 \\
    0 & 1 & 0 \\
    1 & 0 & 0
\end{pmatrix}\right\} \cup \mathcal{A}_1 \cup \mathcal{A}_2 \cup \mathcal{A}_3,\]
where 
\begin{align*}
\mathcal{A}_1 & = 
\left\{
\begin{pmatrix}
    0 & 0 & 1 \\
    1 & 0 & 0 \\
    i & 1 & 0
\end{pmatrix}, 
\begin{pmatrix}
    0 & 1 & 0 \\
    0 & i & 1 \\
    1 & 0 & 0
\end{pmatrix}
\colon 0 \leq i < p
\right\}, \\
\mathcal{A}_2 & =
\left\{
\begin{pmatrix}
    0 & 1 & 0 \\
    1 & 0 & 0 \\
    i & j & 1
\end{pmatrix}, 
\begin{pmatrix}
    1 & 0 & 0 \\
    i & 0 & 1 \\
    j & 1 & 0
\end{pmatrix}
\colon 0 \leq i, j < p
\right\}, \\
\mathcal{A}_3 & = 
\left\{
\begin{pmatrix}
    1 & 0 & 0 \\
    i & 1 & 0 \\
    j & k & 1
\end{pmatrix}
\colon 0 \leq i, j, k < p
\right\}.
\end{align*}

By condition (1) of \cref{def:AGG-3.1}, a function $f \in W(\mathcal{I}(3,p))$ is determined by its values on $\mathcal{A}$.  We construct $M$ such that, for a given $\gamma \in \mathcal{A}$, the rows of $M$ corresponding to $\gamma$ encode conditions (2)-(4) with $m = \gamma$.  We observe that an $\mathcal{I}(3,p)$ invariant function satisfies conditions (2)-(4) for all $m$ if and only if it does for all $\gamma \in A$.  Therefore in practice, we construct three $N \times N$ submatrices.  The submatrices correspond to the separate conditions (2), (3), and (4), i.e.\ the action of $\sigma_1$, $\sigma_2$, and $h$ on elements $\gamma \in \mathcal{A}$. 

\begin{enumerate}
\setcounter{enumi}{1}
    \item Condition (2) specifies that 
        \[f(\gamma) + f(\sigma_1 \gamma) = 0.\]
    To build the matrix $M_{\sigma_1}$, whose rows and columns are indexed by the elements in $\mathcal{A}$, we start by creating a list of pairs $(\gamma, \gamma')$ such that 
    \[\sigma_1 \gamma \in \gamma' \mathcal{I}.\]
    Then, we construct the rows of $M_{\sigma_1}$ by ranging over all $\gamma \in \mathcal{A}$, placing a 1 in the $\gamma$ column and a 1 in the $\gamma'$ column. In the case when $\gamma = \gamma'$, a 2 is placed in the $\gamma$ column.
    \item Condition (3) specifies that 
        \[f(\gamma) - f(\sigma_2 \gamma) = 0. \]
    To build the matrix $M_{\sigma_2}$, we again start by creating a list of pairs $(\gamma, \gamma')$ such that 
    \[\sigma_2 \gamma \in \gamma' \mathcal{I}.\]
    Then, we construct the rows of $M_{\sigma_2}$ by ranging over all $\gamma \in \mathcal{A}$, placing a 1 in the $\gamma$ column and a $-1$ in the $\gamma'$ column. In the case when $\gamma = \gamma'$, a placeholder row of zeroes is added.
    \item Condition (4) specifies that 
        \[f(\gamma) + f(h\gamma) + f(h^2\gamma) = 0.\]
    To build the matrix $M_h$, we first create a list of pairs $(\gamma_0, \gamma_1)$ such that 
    \[h \gamma_0 \in \gamma_1 \mathcal{I}.\]
    We then create a list of pairs $(\gamma_0, \gamma_2)$ such that 
    \[h^2 \gamma_0 \in \gamma_2 \mathcal{I}.\]
    Finally, we construct the rows of $M_h$ by ranging over all $\gamma_0 \in \mathcal{A}$, placing a 1 in the $\gamma_0$ column, a $1$ in the $\gamma_1$ column, and a $1$ in the $\gamma_2$ column. In the case when $\gamma_0 = \gamma_1 = \gamma_2$, a 3 is placed in the $\gamma_0$ column.
\end{enumerate}

By stacking $M_h, M_{\sigma_1}, M_{\sigma_2}$ vertically, we create the matrix $M$ whose kernel is $W(\mathcal{I}(3, p))$.  We then computed the rank of $M$ using the generic \texttt{Rank} function in \textsc{Magma}. From the rank calculation, we could next find the dimension of the kernel, i.e.\ the dimension of $H_3(\mathcal{I}(3,p), \mathbb{C})$, from which we could compute the dimension of $H^3_\text{new}(\mathcal{I}(3,p), \mathbb{C})$ using \cref{eq:iwahori-new}.  For all primes $p \leq 227$, we only found $H^3_\text{new}(\mathcal{I}(3, p), \mathbb{C})$ to have nonzero dimension for $p \in \{19, 43, 127\}$. \cref{table:3} reports the dimensions  of $H_3(\mathcal{I}, \mathbb{C})$, $H^3_\text{cusp}(\mathcal{I}, \mathbb{C})$, and $H^3_\text{new}(\mathcal{I}, \mathbb{C})$ for levels $\mathcal{I}(3, p)$ with nonzero $H^3_\text{new}(\mathcal{I}, \mathbb{C})$ dimension.

\begin{table}
    \small
    \caption{Dimensions of $H_3(\mathcal{I}, \mathbb{C})$, $H^3_\text{cusp}(\mathcal{I}, \mathbb{C})$, and $H^3_\text{new}(\mathcal{I}, \mathbb{C})$ for levels $\mathcal{I}(3, p)$ with nonzero $H^3_\text{new}(\mathcal{I}, \mathbb{C})$ dimension.
    }
    \label{table:3}
    \arraycolsep=0.9em
    \def\arraystretch{1.2}
    \makebox[\textwidth][c]{
    $\begin{array}{c c c c}
        \toprule
        p & \dim H_3(\mathcal{I}(3,p), \mathbb{C}) & \dim H^3_\text{cusp}(\mathcal{I}(3,p), \mathbb{C}) & \dim H^3_\text{new}(\mathcal{I}(3,p), \mathbb{C})\\
        \midrule
        19 & 10 & 3 & 2 \\
        43 & 24 & 5 & 2 \\
        127 & 73 & 12 & 2\\
        \bottomrule
    \end{array}$
    }
\end{table}

We built $M$ over $\mathbb{Q}$ for all prime values $p < 230$.  We opted to work over $\mathbb{Q}$ instead of $\mathbb{C}$, to avoid any floating-point error.  One can similarly work over a large enough finite field, as described in \cite{AGM10}*{Section 2}, to leverage high-efficiency algorithms like \texttt{LinBox}.  However, then additional work is needed to understand the characteristic polynomials that result from the Hecke operator computation, having to lift polynomials over $\mathbb{F}_q$ to polynomials over $\mathbb{C}$ (see \cite{Por25}*{Section 4.2}).  Ultimately, working over a finite field only allowed us to check two additional values of $p$, so we instead focus on the simpler case working over $\mathbb{Q}$. 

In order to compute the Hecke operators, as described in the next section, we
needed an explicit basis for $W(\mathcal{I}(3,p))$. By construction, we could use the basis for the kernel of $M$. For $p \in \{19, 43\}$, we could compute this basis;
however, for $p = 127$, the matrix was too large for effective computation.
As a result, we were unable to produce a basis for level $\mathcal{I}(3, 127)$.


\subsection{Hecke Operators Computation} \label{sec:iwahori-hecke_comp}
With a basis for $W(\mathcal{I}(3,p))$ in hand, we can now focus on computing the action of Hecke operators $E_\ell$ and $F_\ell$ on $H_3(\mathcal{I}(3,p), \mathbb{C})$ (see \cref{sec:prelims-hecke_operators}).  Recall, these specific operators generate the entire Hecke algebra acting on $H_3(\mathcal{I}(3,p), \mathbb{C})$.  To compute the action of $E_\ell$ and $F_\ell$ on $H_3(\mathcal{I}(3,p), \mathbb{C})$, we first need a collection of single coset representatives $B_i$ for $A(\ell)$ and $B(\ell)$ respectively.  We used the following $\ell^2 + \ell + 1$ representatives for $A(\ell)$:
\[L_E \coloneqq
\left\{
\begin{pmatrix}
    1 & 0 & 0 \\
    0 & 1 & 0 \\
    a & b & \ell
\end{pmatrix}
\colon 
0 \leq a, b < \ell
\right\} \cup
\left\{
\begin{pmatrix}
    1 & 0 & 0 \\
    c & \ell & 0 \\
    0 & 0 & 1
\end{pmatrix}
\colon 
0 \leq c < \ell
\right\} \cup 
\left\{
\begin{pmatrix}
    \ell & 0 & 0 \\
    0 & 1 & 0 \\
    0 & 0 & 1
\end{pmatrix}
\right\}.
\]
We used the following $\ell^2 + \ell + 1$ representatives for $B(\ell)$:
\[
L_F \coloneqq
\left\{
\begin{pmatrix}
    1 & 0 & 0 \\
    a & \ell & 0 \\
    b & 0 & \ell
\end{pmatrix}
\colon 
0 \leq a, b < \ell
\right\} \cup
\left\{
\begin{pmatrix}
    \ell & 0 & 0 \\
    0 & 1 & 0 \\
    0 & c & \ell
\end{pmatrix}
\colon 
0 \leq c < \ell
\right\} \cup 
\left\{
\begin{pmatrix}
    \ell & 0 & 0 \\
    0 & \ell & 0 \\
    0 & 0 & 1
\end{pmatrix}
\right\}.
\]
We then followed the method described in \cite{AAC98}*{Lemma 9.1}, instead working with trivial coefficients as opposed to twisted coefficients.  We observe that this is a restatement of \cref{eq:mod-symbols}.

We begin by using the algorithm from \cite{vGvdKTV97}*{Section 2.10} to compute the collection of unimodular symbols $\{[Q_{ij}]\}$ that are homologous to $[B_i^{-1}]$ for each $B_i \in L_E$ (resp.\ $L_F$).  (Note that we are using $[Q_{ij}]$ to denote the unimodular symbol called $[M_{ij}]$ in \cite{AAC98}.)
\begin{remark}
We note that in the statement of \cite{AAC98}*{Lemma 9.1}, $\sum_j [M_{ij}]$ should be homologous to $[B_i^{-1}]$.
\end{remark} 
\noindent We construct a list $L_R$ containing all of the unimodular symbols that stem from these reductions.  One major benefit of this method is that we only need to perform the reduction step once.

Now, for every $Q_{ij} \in L_R$, we construct an $N \times N$ matrix $M_{ij}$.  To build the matrix $M_{ij}$, whose rows and columns are indexed by the elements in $\mathcal{A}$, we start by creating a list of pairs $(\gamma, \gamma')$ such that 
    \[Q_{ij} \gamma \in \gamma' \mathcal{I}.\]
Then, we construct the rows of $M_{ij}$ by ranging over all $\gamma \in \mathcal{A}$, placing a 1 in the corresponding $\gamma'$ column.  We next take the sum $R = \sum_{i,j} M_{ij} \mathcal{B}$, where $\mathcal{B}$ is the basis computed at the end of \cref{sec:iwahori-dim_comp}.  Finally, we solve the equation 
\[\mathcal{B}W = R\]
for $W$.  This gives the characteristic polynomial $\varphi_\ell^{\text{full}}$ for the action of Hecke operators $E_\ell$ (resp.\ $F_\ell$) on $H_3(\mathcal{I}(3,p), \mathbb{C})$.

Recall, the characteristic polynomial $\varphi_\ell^{\text{full}}$ has factors coming from the boundary cohomology and the cuspidal cohomology.  If the calculation from \cref{sec:iwahori-dim_comp} using \cref{eq:iwahori-new} suggested that the dimension of $H^3_\text{new}(\mathcal{I}(3,p), \mathbb{C})$ was nonzero, which was the case for $p \in \{19, 43, 127\}$, we then computed $\varphi_\ell^{\text{full}}$ for $\ell \in \{2, 3, 5, 7, 11\}$.  From the data for small $\ell$, we could isolate the factor $\varphi_\ell$ that corresponds to the action of $E_\ell$ (resp.\ $F_\ell$) on $H^3_\text{new}(\mathcal{I}, \mathbb{C})$, using the results from \cref{sec:iwahori-boundary,sec:iwahori-cusp}.  As a sanity check, we can verify that the splitting field generated by the eigenvalues $e_\ell$ (resp.\ $f_\ell$), i.e.\ the roots of $\varphi_\ell$, are either totally real or CM, which we know must be the case by \cite{APT91}*{Lemma 1.3}.  We now give an example of this process. 

\begin{example}
    For $p = 19$ and $\ell = 5$, the factored polynomial that results from acting the Hecke operator $E_5$ on $H_3(\mathcal{I}(3, 19), \mathbb{C})$ is 
    \[\varphi_\ell^\text{full} = (T - 31)(T - 28)^3(T - 16)^3 (T - 4) (T^2 - 2T + 16).\]
    Consulting the $L$-functions and modular forms database \cite{LMFDB}, we find a  self-dual cusp form of dimension $1$, weight $2$, and level $\Gamma_0(2, 19)$, [\href{https://www.lmfdb.org/ModularForm/GL2/Q/holomorphic/19/2/a/a/}{\texttt{19.2.a.a}}], with $a_5 = 3$.
    
    Using the results from \cref{sec:iwahori-boundary}, we find the trivial factor by computing 
    \[\ell^2 + \ell + 1 = 5^2 + 5 + 1 = 31.\]
    We can then compute $e_\ell$ for $\rho_1$ as 
    \[\ell a_\ell + 1 = 5(3) + 1 = 16.\]
    We finally compute $e_\ell$ for $\rho_2$ as 
    \[a_\ell + \ell^2 = 3 + 5^2 = 28.\]
    Therefore, the boundary factors are
    \[(T - 31)(T - 28)^3(T - 16)^3.\]
    Next, we can use the results from \cref{sec:iwahori-cusp} to determine $e_\ell$ for the symmetric square lifting $\rho_3$ to be 
    \[a_\ell^2 - \ell = 3^2 - 5 = 4.\]
    We note that there is no contribution from level $\Gamma_0(3,p)$, as the dimension of $H^3_\text{cusp}(\Gamma_0(3, 19), \mathbb{C})$ is zero. Therefore, the additional factors coming from the cuspidal cohomology are 
    \[(T-4).\]
    Removing these factors from $\varphi_\ell^\text{full}$, we are left with $\varphi_\ell$, the characteristic polynomial of $E_\ell$ on $H^3_\text{new}(\mathcal{I}(3, 19), \mathbb{C})$:
    \[\varphi_\ell = T^2 - 2T + 16.\]
    \hfill $\blacksquare$
\end{example} 

Once the new cuspidal factor has been identified, we can then begin computing directly on $H^3_\text{new}(\mathcal{I}(3,p), \mathbb{C})$. 
For the two levels on which we could perform Hecke operator computations, $\mathcal{I}(3, 19)$ and $\mathcal{I}(3, 43)$, \cref{table:4} reports the eigenvalues $e_\ell$ of $E_\ell$ for a fixed common eigenvector, the corresponding characteristic polynomials $\varphi_\ell$ of $E_\ell$ on $H^3_{\text{new}}(\mathcal{I}(3, p), \mathbb{C})$, and the field generated by the eigenvalues for $\ell < 70$.

\begin{table}
    \small
    \caption{Hecke eigenvalues $e_\ell$ for Hecke operators $E_\ell$ for a fixed common eigenvector and corresponding characteristic polynomials $\varphi_\ell$ on $H^3_{\text{new}}(\mathcal I(3,p), \mathbb{C})$.}
    \label{table:4}
    \arraycolsep=0.9em
    \def\arraystretch{1.2}
    $\begin{array}{ r | r l | r l }
        \toprule
        & \multicolumn{2}{c |}{p = 19} & \multicolumn{2}{c}{p = 43} \\
        & k \simeq \mathbb{Q}(\sqrt{-15}) \hfil &  \hfil \omega = \frac{1 + \sqrt{-15}}{2} & k \simeq \mathbb{Q}(\sqrt{-7}) \hfil & \hfil \omega = \frac{1 + \sqrt{-7}}{2}  \\
        \midrule
        \ell & e_\ell \hfil & \hfil \varphi_\ell & e_\ell \hfil & \hfil \varphi_\ell \\
        \midrule
        2 & \omega - 1 & T^2 + T + 4 & \omega - 2 & T^2 + 3T + 4 \\
        3 & -2\omega - 1 & T^2 + 4T + 19 & 1 & (T-1)^2 \\
        5 & -2\omega + 2 & T^2 - 2T + 16 & -4\omega - 1 & T^2 + 6T + 37 \\
        7 & 12 & (T-12)^2 & 7 & (T-7)^2 \\
        11 & -2 & (T+2)^2 & -6\omega + 8 & T^2 - 10T + 88 \\
        13 & 3 & (T-3)^2 & 12 & (T-12)^2 \\
        17 & -2\omega + 8 & T^2 - 14T + 64 & -12\omega - 2 & T^2 + 16T + 316 \\
        19 & -*- & -*- & -12\omega - 9 & T^2 + 30T + 477 \\
        23 & 8\omega - 21 & T^2 + 34T + 529 & 10\omega - 20 & T^2 + 30T + 400 \\
        29 & 2\omega - 15 & T^2 + 28T + 211 & -7 & (T+7)^2 \\
        31 & 3 & (T-3)^2 & -6\omega + 36 & T^2 - 66T + 1152 \\
        37 & 18\omega + 3 & T^2 - 24T + 1359 & -16\omega + 13 & T^2 - 10T + 473 \\
        41 & 10\omega + 23 & T^2 - 56T + 1159 & 20\omega + 2 & T^2 - 24T + 844 \\
        43 & 36\omega - 60 & T^2 + 84T + 6624 & -*- & -*- \\
        47 & -10\omega - 18 & T^2 + 46T + 904 & 8\omega - 7 & T^2 + 6T + 121 \\
        53 & -4\omega + 45 & T^2 - 86T + 1909 & 8\omega + 80 & T^2 - 168T + 7168 \\
        59 & 16\omega - 7 & T^2 - 2T + 961 & 24\omega + 17 & T^2 - 58T + 1849 \\
        61 & 48 & (T-48)^2 & 24\omega - 27 & T^2 + 30T + 1233 \\
        67 & 57 & (T-57)^2 & 14\omega - 2 & T^2 - 10T + 368 \\
        \bottomrule
    \end{array}$
\end{table}

Let $k$ denote the splitting field generated by the eigenvalues. For level $\mathcal{I}(3, 19)$, $k \simeq \mathbb{Q}(\sqrt{-15})$.  Similar to the observations made in the $\Gamma_0(3,p)$ setting in \cite{AGG84}*{Section 6}, we find that the primes that ramify in $\mathbb{Q}(\sqrt{-15})$ are $3$ and $5$, and $19$ is a square in both $\mathbb{Z}_3$ and $\mathbb{Z}_5$.  For level $\mathcal{I}(3, 43)$, $k \simeq \mathbb{Q}(\sqrt{-7})$.  We note that the only prime that ramifies in $\mathbb{Q}(\sqrt{-7})$ is $7$, and $43$ is a square in $\mathbb{Z}_7$.
Also, in line with the $\Gamma_0(3,p)$ observations, we see that the ring spanned by the images of Hecke operators in $k$ is $\mathbb{Z} +\mathbb{Z}\sqrt{D}$, where $D$ denotes the negative square free rational integer such that $k \simeq \mathbb{Q}(\sqrt{D})$.

We conclude by noting that the dimension of $H^3_{\text{new}}(\mathcal I(3, p), \mathbb{C})$ is two for both levels at which previously-unseen cuspidal classes appear.  There is no guarantee that the dimension remains constant as $p$ grows. 

\section*{Acknowledgments}
The author thanks David Pollack for helpful discussions during the preparation of the article.  The author also thanks Chris Williams for prompting this line of inquiry.  Finally, the author thanks Wesleyan University for computer time supported by the NSF under grant numbers CNS-0619508 and CNS-0959856.


\begin{bibdiv}
    \begin{biblist}

    \bib{AAC98}{article}{
       author={Allison, Gerald},
       author={Ash, Avner},
       author={Conrad, Eric},
       title={Galois representations, Hecke operators, and the mod-$p$
       cohomology of ${\rm GL}(3,\mathbf Z)$ with twisted coefficients},
       journal={Experiment. Math.},
       volume={7},
       date={1998},
       number={4},
       pages={361--390},
       issn={1058-6458},
       review={\MR{1678079}},
    }

    \bib{AGG84}{article}{
          author={Ash, Avner},
          author={Grayson, Daniel},
          author={Green, Philip},
           title={Computations of cuspidal cohomology of congruence subgroups of {${\rm SL}(3,{\bf Z})$}},
            date={1984},
            ISSN={0022-314X,1096-1658},
         journal={J. Number Theory},
          volume={19},
          number={3},
           pages={412\ndash 436},
             url={https://doi.org/10.1016/0022-314X(84)90081-7},
          review={\MR{769792}},
    }

    \bib{AGM10}{article}{
          author={Ash, Avner},
          author={Gunnells, Paul~E.},
          author={McConnell, Mark},
           title={Cohomology of congruence subgroups of {${\rm SL}_4(\bf Z)$}. {III}},
            date={2010},
            ISSN={0025-5718,1088-6842},
         journal={Math. Comp.},
          volume={79},
          number={271},
           pages={1811\ndash 1831},
             url={https://doi.org/10.1090/S0025-5718-10-02331-8},
          review={\MR{2630015}},
    }

    \bib{APT91}{article}{
       author={Ash, Avner},
       author={Pinch, Richard},
       author={Taylor, Richard},
       title={An $\widehat{A_4}$ extension of ${\bf Q}$ attached to a
       nonselfdual automorphic form on ${\rm GL}(3)$},
       journal={Math. Ann.},
       volume={291},
       date={1991},
       number={4},
       pages={753--766},
       issn={0025-5831},
       review={\MR{1135542}},
    }

    \bib{AR79}{article}{
          author={Ash, Avner},
          author={Rudolph, Lee},
           title={The modular symbol and continued fractions in higher dimensions},
            date={1979},
            ISSN={0020-9910,1432-1297},
         journal={Invent. Math.},
          volume={55},
          number={3},
           pages={241\ndash 250},
             url={https://doi.org/10.1007/BF01406842},
          review={\MR{553998}},
    }
    \bib{AS86}{article}{
       author={Ash, Avner},
       author={Stevens, Glenn},
       title={Cohomology of arithmetic groups and congruences between systems of
       Hecke eigenvalues},
       journal={J. Reine Angew. Math.},
       volume={365},
       date={1986},
       pages={192--220},
       issn={0075-4102},
       review={\MR{0826158}},
       doi={10.1515/crll.1986.365.192},
    }

    \bib{AT99}{article}{
       author={Ash, Avner},
       author={Tiep, Pham Huu},
       title={Modular representations of ${\rm GL}(3,{\bf F}_p)$, symmetric
       squares, and mod-$p$ cohomology of ${\rm GL}(3,{\bf Z})$},
       journal={J. Algebra},
       volume={222},
       date={1999},
       number={2},
       pages={376--399},
       issn={0021-8693},
       review={\MR{1727178}},
       doi={10.1006/jabr.1999.7988},
    }

    \bib{BSGW25}{misc}{
        author = {Barrera Salazar, Daniel},
        author = {Graham, Andrew},
        author = {Williams, Chris},
        Title = {Local-global compatibility and the exceptional zero conjecture for $\mathrm{GL}(3)$},
        Year = {2025},
        Eprint = {arXiv:2508.10225},
        note = {\href{https://arxiv.org/abs/2508.10225}{\tt arXiv:2508.10225}}
        }

    \bib{BS73}{article}{
       author={Borel, A.},
       author={Serre, J.-P.},
       title={Corners and arithmetic groups},
       journal={Comment. Math. Helv.},
       volume={48},
       date={1973},
       pages={436--491},
       issn={0010-2571},
       review={\MR{0387495}},
       doi={10.1007/BF02566134},
    }

    \bibitem[Magma]{Magma}
    W. Bosma, J. J. Cannon, C. Fieker, A. Steel (eds.), \textit{Handbook of Magma functions}, Version 2.28-14 (2024). \href{https://magma.maths.usyd.edu.au/magma/handbook/}{\tt https://magma.maths.usyd.edu.au/magma/handbook/}.

    \bib{Clo90}{article}{
   author={Clozel, Laurent},
   title={Motifs et formes automorphes: applications du principe de
   fonctorialit\'e},
   language={French},
   conference={
      title={Automorphic forms, Shimura varieties, and $L$-functions, Vol.\
      I},
      address={Ann Arbor, MI},
      date={1988},
   },
   book={
      series={Perspect. Math.},
      volume={10},
      publisher={Academic Press, Boston, MA},
   },
   isbn={0-12-176651-9},
   date={1990},
   pages={77--159},
   review={\MR{1044819}},
}

    \bib{GJ76}{article}{
       author={Gelbart, Stephen},
       author={Jacquet, Herv\'e},
       title={A relation between automorphic forms on ${\rm GL}(2)$ and ${\rm
       GL}(3)$},
       journal={Proc. Nat. Acad. Sci. U.S.A.},
       volume={73},
       date={1976},
       number={10},
       pages={3348--3350},
       issn={0027-8424},
       review={\MR{0412156}},
    }
    
    \bib{Gro96}{article}{
   author={Gross, Benedict H.},
   title={On the Satake isomorphism},
   conference={
      title={Galois representations in arithmetic algebraic geometry},
      address={Durham},
      date={1996},
   },
   book={
      series={London Math. Soc. Lecture Note Ser.},
      volume={254},
      publisher={Cambridge Univ. Press, Cambridge},
   },
   isbn={0-521-64419-4},
   date={1998},
   pages={223--237},
   review={\MR{1696481}},
   doi={10.1017/CBO9780511662010.006},
}

    \bib{Gun00}{article}{
       author={Gunnells, Paul E.},
       title={Modular symbols and Hecke operators},
       conference={
          title={Algorithmic number theory},
          address={Leiden},
          date={2000},
       },
       book={
          series={Lecture Notes in Comput. Sci.},
          volume={1838},
          publisher={Springer, Berlin},
       },
       isbn={3-540-67695-3},
       date={2000},
       pages={347--357},
       review={\MR{1850616}},
    }

    \bib{Har91}{article}{
       author={Harder, G\"unter},
       title={Eisenstein cohomology of arithmetic groups and its applications to
       number theory},
       conference={
          title={Proceedings of the International Congress of Mathematicians,
          Vol.\ I, II},
          address={Kyoto},
          date={1990},
       },
       book={
          publisher={Math. Soc. Japan, Tokyo},
       },
       isbn={4-431-70047-1},
       date={1991},
       pages={779--790},
       review={\MR{1159264}},
    }

    \bib{LS82}{article}{
          author={Lee, Ronnie},
          author={Schwermer, Joachim},
           title={Cohomology of arithmetic subgroups of {${\rm SL}\sb{3}$} at infinity},
            date={1982},
            ISSN={0075-4102,1435-5345},
         journal={J. Reine Angew. Math.},
          volume={330},
           pages={100\ndash 131},
             url={https://doi.org/10.1515/crll.1982.330.100},
          review={\MR{641814}},
    }

    \bib{Por25}{article}{
        author = {Porat, Zachary},
        title = {Computations directly on the cuspidal cohomology of congruence subgroups of $\mathrm{SL}(3, \mathbb{Z})$},
        date = {2025-11-12},
        journal = {Math. Comp.},
        url={https://doi.org/10.1090/mcom/4155},
        doi={10.1090/mcom/4155}
    }

    \bib{Iwahori-Code}{misc}{
          author={Porat, Zachary},
           title={GitHub repository for $\mathcal{I}(3,p)$ related scripts and data},
            note={Available at \href{https://github.com/zporat/GL3-Iwahori-C}{\texttt{https://github.com/zporat/GL3-Iwahori-C}} (updated 2026)},
            label={Por26}
    }

    \bib{LMFDB}{misc}{
      author={{The LMFDB Collaboration}},
       title={The {L}-functions and modular forms database},
        note={Available online at \href{https://www.lmfdb.org}{\tt https://www.lmfdb.org}},
        label = {LMFDB},
    }
    
    \bib{vGvdKTV97}{article}{
          author={van Geemen, Bert},
          author={van~der Kallen, Wilberd},
          author={Top, Jaap},
          author={Verberkmoes, Alain},
           title={Hecke eigenforms in the cohomology of congruence subgroups of {${\rm SL}(3,{\bf Z})$}},
            date={1997},
            ISSN={1058-6458,1944-950X},
         journal={Experiment. Math.},
          volume={6},
          number={2},
           pages={163\ndash 174},
             url={http://projecteuclid.org/euclid.em/1047650002},
          review={\MR{1474576}},
    }
    \end{biblist}
\end{bibdiv}
\end{document}